\documentclass[a4,11pt]{amsart}

\usepackage{amssymb}
\usepackage{stmaryrd}
\usepackage[all]{xy}
\usepackage{mathrsfs}
\usepackage{color}
\usepackage{geometry}
\usepackage{enumitem}
\usepackage{color}

\numberwithin{itemcounter}{subsection}

\def\k{\mathbf{k}}
\def\C{\mathbb{C}}
\def\d{\mathbf{d}}
\def\H{\mathbf{H}}
\def\uQ{\underline{Q}}

\theoremstyle{plain}

\newtheorem{theorem}{Theorem}[section]
\newtheorem{lemma}[theorem]{Lemma}
\newtheorem{lemma-definition}[theorem]{Lemma-Definition}
\newtheorem{definition-lemma}[theorem]{Definition-Lemma}

\newtheorem{proposition}[theorem]{Proposition}
\newtheorem{conjecture}[theorem]{Conjecture}
\newtheorem{corollary}[theorem]{Corollary}
\theoremstyle{definition}

\theoremstyle{remark}

\numberwithin{equation}{section}

\author{T. Bozec, O. Schiffmann}

\title{Counting absolutely cuspidals for quivers }

\begin{document}

\begin{abstract} For an arbitrary quiver $Q=(I,\Omega)$ and dimension vector $\mathbf{d} \in \mathbb{N}^I$ we define the dimension of \textit{absolutely cuspidal functions} on the moduli stacks of representations of dimension $\mathbf{d}$ of a quiver $Q$ over a finite field $\mathbb{F}_q$,  and prove that it is a polynomial in $q$, which we conjecture to be positive and integral. We obtain a closed formula for these dimensions of spaces of cuspidals for totally negative quivers. 
\end{abstract}

\maketitle

\setcounter{tocdepth}{2}

\tableofcontents

\section{Introduction}

\noindent
\textbf{1.1.} Let $Q=(I,\Omega)$ be an arbitrary quiver. For any finite field $\k=\mathbb{F}_q$ one may consider the Hall algebra $\mathbf{H}_{Q/ \k}$, which is a (twisted) self-dual complex Hopf algebra naturally graded by the lattice $\mathbb{Z}^I$~:
$$\mathbf{H}_{Q/\k}=\bigoplus_{\mathbf{d} \in \mathbb{N}^I} \mathbf{H}_{Q/\k}[\mathbf{d}].$$ 
Denote by $(\epsilon_i)_i$ the canonical basis of $\mathbb{Z}^I$.
When the quiver $Q$ contains no edge loop the spaces $\mathbf{H}_{Q/\k}[\epsilon_i]$ are all one-dimensional. In that case, one often considers the \textit{spherical subalgebra} $\mathbf{H}_{Q/\k}^{sph} \subset \mathbf{H}_{Q/\k}$ which is defined as the subalgebra generated by
$\mathbf{H}_{Q/\k}[\epsilon_i]$, $i \in I$.
By the Ringel-Green theorem, $\mathbf{H}_{Q/\k}^{sph}$ is isomorphic to the specialization at $v=q^{\frac{1}{2}}$ of the positive half  $\mathbf{U}_v^+(\mathfrak{g}_Q)$ of the quantized envelopping algebra of the Kac-Moody algebra $\mathfrak{g}_Q$ associated to $Q$. In particular the character
$$\text{ch}(\mathbf{H}_{Q/\k}^{sph})=\sum_{\mathbf{d}} \text{dim}(\mathbf{H}_{Q/\k}^{sph}[\mathbf{d}]) z^{\mathbf{d}}$$
is given by the Weyl-Kac formula, and is independent of $q$. On the contrary, the character of $\mathbf{H}_{Q/\k}$ does strongly depend on $q$ as soon as $Q$  is not of finite Dynkin type (i.e. as soon as $\mathbf{H}_{Q/\k}^{sph}$ is a proper subalgebra of $\mathbf{H}_{Q/\k}$).
To understand this dependence, let us consider the subspace of \textit{primitive} or \textit{cuspidal} elements  
$$\mathbf{H}_{Q/\k}^{cusp}[\mathbf{d}]=\{x \in \mathbf{H}_{Q/\k}[\mathbf{d}]\;|\; \Delta(x) =x\otimes 1 + 1 \otimes x\}.$$
It is well-known and easy to see that $\{\mathbf{H}_{Q/\k}^{cusp}[\mathbf{d}]\;|\; \mathbf{d} \in \mathbb{N}^I\}$ form a minimal set of
generators for $\mathbf{H}_{Q/\k}$.

\vspace{.1in}

Set $\Pi_0=\{\epsilon_i\;|\; i \in I \}$ and let $(\,,\,)~: \mathbb{Z}^I \otimes \mathbb{Z}^I \to \mathbb{Z}$ be the symmetrized Euler form.
The starting point of this short note is the following theorem, due to B. Deng and J. Xiao (in the case of (possibly valued) quivers without oriented cycles), see
\cite{DX}~:

\vspace{.1in}

\begin{theorem}\label{T:main} For any quiver $Q=(I,\Omega)$ and dimension vector $\mathbf{d} \in \mathbb{N}^I$ there exists a polynomial
$C_{Q,\mathbf{d}} \in \mathbb{Q}[t]$
such that for any finite field $\k$
$$\mathrm{dim}(\mathbf{H}^{cusp}_{Q/\k}[\mathbf{d}])=C_{Q,\mathbf{d}}(|\k|).$$
This polynomial does not depend on the orientation $\Omega$. 
\end{theorem}

\vspace{.1in}

The proof we give in \textbf{4.1.} (which is close to the original proof in \cite{DX}) is constructive, i.e. it also provides a recursive algorithm to compute $C_{Q,\mathbf{d}}$. In particular, using this algorithm one can prove the following

\vspace{.1in}

\begin{proposition}\label{P:firstprop} We have $C_{Q,\d}=0$ unless $\d \in \Pi_0$ or $(\d,\d)\leq 0$ and
in the latter case we have $C_{Q,\d}(0)=0$.
\end{proposition}

\vspace{.1in}

It is easy to see that the coefficients of $C_{Q,\d}$ are in general neither integral nor positive. The aim of this note is to find a correct substitute for $C_{Q,\d}$.

\vspace{.1in}

Recall that, by a theorem of Kac \cite{Kac}, there exist polynomials $A_{Q,\mathbf{d}}, I_{Q,\mathbf{d}} \in \mathbb{Q}[t]$ such that for any $\k$ we have
$$A_{Q,\mathbf{d}}(|\k|)=\#\text{Ind}^{geom}_{Q,\mathbf{d},\k}, \qquad I_{Q,\mathbf{d}}(|\k|)=\#\text{Ind}_{Q,\mathbf{d},\k}$$
where $\text{Ind}_{Q,\mathbf{d},\k}, \text{Ind}^{geom}_{Q,\mathbf{d},\k}$ is the set of indecomposable, resp. geometrically indecomposable, representations of $Q$ over $\k$ of dimension $\mathbf{d}$. The two families of polynomials are related by the following equality of generating series involving plethystic exponentials (see Section~\textbf{1.5.})
$$\textup{Exp}_{t,z}\left( \sum_{\mathbf{d}} A_{\mathbf{d}}(t)z^{\mathbf{d}}\right)=\textup{Exp}_z\left(\sum_{\mathbf{d}}I_{\mathbf{d}}(t)z^{\mathbf{d}}\right)$$
(see \cite{Hua}).
Since the character of $\mathbf{H}_{Q/\k}$ clearly depends on the number of isomorphism classes of representations of $Q$ over $\k$ of fixed dimension, it is clear that $C_{Q,\mathbf{d}}$ is strongly related to $I_{Q,\mathbf{d}}$. In general, this relation is somewhat intricate, relying on the Borcherds denominator formula for generalized Kac-Moody algebras, see Sections~\textbf{2.2.} and \textbf{4.1.} The coefficients of the polynomial $C_{Q,\mathbf{d}}$ are in general neither positive nor integral, similarly to the Kac polynomials $I_{Q,\mathbf{d}}$. It is well-known that the better behaved Kac polynomial is the $A$-polynomial $A_{Q,\mathbf{d}}$, which is known to belong to $\mathbb{N}[t]$ (see \cite{HLRV}) and
to have some cohomological interpretation in terms of Donaldson-Thomas invariants (see \cite{Moz}).
Applying the same combinatorial process as above, but to the Kac A-polynomials $A_{Q,\mathbf{d}}$ instead of the Kac $I$-polynomials and based on a \textit{graded} analog of the Borcherds character formula yield some rational polynomials $C^{abs}_{Q,\mathbf{d}} \in \mathbb{Q}[t]$, see Section~\textbf{4.2.}  Equivalently, we may also directly define $C^{abs}_{Q,\mathbf{d}}(t)$ via the following relations~:
$$C^{abs}_{Q,\mathbf{d}}(t)=C_{Q,\mathbf{d}}(t)$$
for all $\mathbf{d} \in \mathbb{N}^I$ such that $(\mathbf{d},\mathbf{d})<0$ and
$$\textup{Exp}_z\left(\sum_{l \geq 1} C_{Q,l\mathbf{d}}(t)z^{l}\right)=\textup{Exp}_{t,z}\left(\sum_{l \geq 1} C^{abs}_{Q,l\mathbf{d}}(t)z^{l}\right)$$
 for all $\mathbf{d} \in (\mathbb{N}^I)_{prim}$ such that $(\mathbf{d},\mathbf{d})=0.$ In the above, $(\mathbb{N}^I)_{prim}$ is the set of indivisible elements of $\mathbb{N}^I$.

\vspace{.1in}

We make the following conjecture~:

\begin{conjecture}\label{C:main} For any $Q=(I,\Omega)$ and $\mathbf{d}$ we have $C^{abs}_{Q,\mathbf{d}}\in \mathbb{N}[t]$.
\end{conjecture}

\vspace{.1in}

As a first step towards the above conjecture, we prove (see Section~\textbf{4.4})~:

\begin{theorem}\label{T:main2} For any $Q$ and any $\mathbf{d}\in \mathcal{C}_{Q/\k}$ we have $C^{abs}_{Q,\mathbf{d}}(t) \in \mathbb{Z}[t]$.
\end{theorem}

\vspace{.1in}

Theorem~\ref{T:main} is a quiver analog of a conjecture of Deligne and Kontsevich which asserts that the dimension of the space of geometrically cuspidal functions on $Bun_{GL_{r,d}}(X)$ for $X$ a smooth projective curve defined over a finite field $\k$, is given by a polynomial in the Weil numbers of $X$ (which only depends on $r$), see \cite{Deligne}, \cite{Kont}\footnote{More precisely, the conjecture of Deligne and Kontsevich is phrased in terms of counting geometrically irreducible $\overline{\mathbb{Q}_l}$-representations of $\pi_1^{arith}(X)$; the two problems are equivalent by the Langlands correspondence, \cite{Lafforgue}}. Note that in the context of curves there is, via the Langlands correspondence, a geometric way of characterizing absolutely (or geometrically) cuspidal functions among all cuspidal functions; since there is yet no analogue of the (geometric) Langlands correspondence in the context of quivers, we have no choice but to define the numbers $C^{abs}_{Q,\mathbf{d}}$ in the above slightly \textit{ad-hoc} combinatorial way.

\vspace{.1in}

Following Deligne, we may also make the following somewhat imprecise conjecture~:

\vspace{.1in}

\begin{conjecture}\label{C:main2} For any $Q=(I,\Omega)$ and $\mathbf{d}$ there exists a 'natural' algebraic variety $\mathcal{X}_{Q,\mathbf{d}}$ defined over $\mathbb{Z}$ such that we have $C^{abs}_{Q,\mathbf{d}}(|\k|)=\#\mathcal{X}_{Q,\mathbf{d}}(\k)$ for any finite field $\k$.
\end{conjecture}

The above conjecture is rather optimistic; a slightly less optimistic conjecture would claim the existence of such a variety $\mathcal{X}_{Q,\mathbf{d},p}$ for any fixed characteristic $p>0$, or even for any large enough fixed characteristic $p$.

\vspace{.2in}

\noindent
\textbf{1.2.} \textit{Examples.} Let us give some simple examples of polynomials $C^{abs}_{Q,\mathbf{d}}, C_{Q,\mathbf{d}}$.  As usual, we denote by
$\langle \;,\; \rangle: \mathbb{Z}^I \times \mathbb{Z}^I \to \mathbb{Z}$ the (nonsymmetrized) Euler form of $Q$. 
\begin{enumerate}
\item[i)] If $Q$ is a finite type quiver then $C_{Q,\mathbf{d}}=C^{abs}_{Q,\mathbf{d}}=1$ if $\mathbf{d} \in \{\epsilon_i\;|\;i \in I\}$ and $C_{Q,\mathbf{d}}=C^{abs}_{Q,\mathbf{d}}=0$ otherwise,

\item[ii)] If $Q$ is a Jordan quiver then $C^{abs}_{Q,d}=t$ for any $d \geq 1$ while the polynomials $C_{Q,d}$ are determined by the following generating series~:
$$\sum_d C_{Q,d} z^d=\sum_{l,n \geq 1} \frac{\mu(l)}{ln} \frac{t^nz^{ln}}{1-z^{ln}},$$
where $\mu$ is the M\"obius function\footnote{in terms of plethystic functions, we may also write $\sum_d C_{Q,d} z^d=\textup{Log}_z\textup{Exp}_{t,z}\left(tz(1-z)^{-1}\right)$}.

\item[iii)] If $Q$ is an affine Dynkin diagram, with indivisible imaginary root $\delta$, then $C^{abs}_{Q,\mathbf{d}}=C_{Q,\mathbf{d}}=1$ if $\mathbf{d}\in \{\epsilon_i\;|\; i \in I\}$, $C^{abs}_{Q,l\delta}=t$ for any $l \geq 1$; the polynomials $C_{Q,l\delta}$ are again determined by the generating series~:
$$\sum_{d} C_{Q,d\delta} z^d=\sum_{l,n \geq 1} \frac{\mu(l)}{ln} \frac{t^nz^{ln}}{1-z^{ln}},$$
and
$C^{abs}_{Q,\mathbf{d}}=C_{Q,\mathbf{d}}=0$ if $\mathbf{d} \not\in \mathbb{N}\delta \cup \{\epsilon_i\;|\; i \in I\}$.

\item[iv)] If $Q$ is the $3$-Kronecker quiver, i.e. $I=\{1,2\}$ and $\Omega=\{h_1, h_2, h_3\}$ with $h_i : 1 \mapsto 2$ for all $i$, then
$C_{Q,\mathbf{d}}=C^{abs}_{Q,\mathbf{d}}$ for all $\mathbf{d}$ and furthermore
$$C_{Q,(1,1)}=t^2+t, \qquad C_{Q,(2,2)}=t^5+t^4+2t^3+t^2+t, \qquad C_{Q,(2,3)}=t^6+t^4+t^2,$$
$$C_{Q,(3,3)}=t^{10} + t^{9} + 3t^8 + 4t^7 + 4t^6 + 6t^5 + 6t^4 + 4t^3+32t^2+31t.$$

\item[v)]If $Q=S_g$ is the quiver with one vertex and $g$ loops, with $g >1$ then $C^{abs}_{Q,d}=C_{Q,d}$ for all $d$, and are determined by the following generating series~:
\begin{equation*}
\begin{split}
\textup{Log}_{t,z}&\left(1-\sum_d C^{abs}_{Q,d} z^d\right)\\
&=(t-1)\textup{Log}_{t,z}\left(\sum_{\lambda} \left[t^{(g-1)\sum_k \lambda^2_k} \prod_{k} [\infty,\lambda_k-\lambda_{k+1}]_{t^{-1}}z^{|\lambda|}\right]\right)
\end{split}
\end{equation*}
where the sum ranges over all partitions $\lambda$, and $[\infty,a]_u=\prod_{l=1}^a(1-u^l)^{-1}$ is the usual Pochammer symbol.
In particular, we have
\begin{align*}
C^{abs}_{S_g,1}&=t^g\\
C^{abs}_{S_g,2}&=t^{2g-1}\frac{t^{2g}-1}{t^2-1}\\
C^{abs}_{S_g,3}&=\frac{t^{9g-3}-t^{5g+2}-t^{5g-2}-t^{5g-3}+t^{3g+2}+t^{3g-2}}{(t^2-1)(t^3-1)}.
\end{align*}

\item[vi)] In general, we have $C^{abs}_{Q,\mathbf{d}}=C_{Q,\mathbf{d}}=t^{g_i}$ if $\mathbf d \in \{\epsilon_i\;|\; i \in I\}$, and $g_i$ is the number of loops at $i$; in addition, $C^{abs}_{Q,\mathbf{d}}=C_{Q,\mathbf{d}}=0$ if $\mathbf{d}$ is not an imaginary weight (i.e. if $\langle\mathbf{d}, \mathbf{d}\rangle >0$), unless $\d \in \{\epsilon_i\;|\; i \in I\}$.
\end{enumerate}

\vspace{.1in}

In cases ii) and iii) above we may take $\mathcal{X}_{Q,l}$ and $\mathcal{X}_{Q,l\delta}$ equal to $\mathbb{A}^1$ (for any $l$). Case ii) is
classical, case iii) is due to Hua-Xiao, see \cite[Prop. 5.1.1]{HX} while case iv) is due to Deng-Xiao, see \cite[4.2, Ex. 2]{DX}.

\vspace{.1in}

As mentioned above, the relation between $C_{Q,\mathbf{d}}, C^{abs}_{Q,\mathbf{d}}$ and the Kac polynomials is somewhat subtle. There is however one special case in which this relation is easily expressed. Call a quiver $Q=(I,\Omega)$ \textit{totally negative} if the symmetrized Euler form $(\;,\;)$ is totally negative, i.e. if $( \d,\mathbf{n}) <0$ for all $\d,\mathbf{n} \in \mathbb{N}^I \backslash\{0\}$. It is equivalent to the condition that any two vertices are connected by an edge, and that any vertex carries at least two edge loops. For instance, quivers $S_g$ for $g >1$ are totally negative.
 
\vspace{.1in}

\begin{theorem}\label{P:totneg} Let $Q$ be a totally negative quiver. Then we have $C_{Q,\mathbf{d}}(t)=C^{abs}_{Q,\mathbf{d}}(t)$ for any $\mathbf{d} \in \mathbb{N}^I$, and the following relation of generating series holds~:
\begin{equation}\label{E:totneg}
1- \sum_{\mathbf{d}>0} C_{Q,\mathbf{d}}(t) z^{\d}=\mathrm{Exp}_{t,z}\bigg(-\sum_{\mathbf{d}>0} A_{Q,\mathbf{d}}(t)z^{\d}\bigg).
\end{equation}
In particular, we have $C_{Q,\d}^{abs}(t) \in \mathbb{Z}[t]$ for any $\d$.
\end{theorem}

\vspace{.1in}

Example iv) above is a direct application of (\ref{E:totneg}) together with Hua's explicit formula for Kac polynomials, see \cite{Hua}. For a general quiver,
injecting Hua's formula in (\ref{E:totneg}) yields a rather cumbersome formula. In any case, we may rephrase (\ref{E:totneg}) in terms of plethystic exponential
as a closed formula for $C^{abs}_{Q,\d}$~:

\begin{corollary}\label{C:totneg} Assume that $Q$ is totally negative. Then for any $\mathbf{d}$ we have
$$C_{Q,\mathbf{d}}(t)=\sum_{p} (-1)^{-1+\sum_{l,\mathbf{n}} p(l,\mathbf{n})} \prod_{l,\mathbf{n}} \frac{(A_{\mathbf{n}}(t^l))^{p(l,\mathbf{n})}}{p(l,\mathbf{n})!\;l^{p(l,\mathbf{n})}}$$
where the sum ranges over all maps $p: \mathbb{N} \times \mathbb{N}^I \to \mathbb{N}$ such that $\sum_{l,\mathbf{n}} p(l,\mathbf{n})l\mathbf{n}=\mathbf{d}$.
\end{corollary}

\vspace{.2in}

\noindent
\textbf{1.3.} \textit{Relation to Okounkov's conjecture, Maulik-Okounkov Lie algebras and the Davison-Meinhardt Lie algebra of BPS states.} The (rather straightforward) proof of Theorem~\ref{T:main} relies on the fact, proved by Sevenhant and Van den Bergh (\cite{SVdB}), that $\mathbf{H}_{Q/\k}$ is isomorphic to a suitable specialization of the positive half $\mathbf{U}_{\mathbf{v}}^+(\widetilde{\mathfrak{g}}_{Q/\k})$ of the quantized envelopping algebra of a huge \textit{generalized} or \textit{Borcherds} \textit{Kac-Moody algebra}
$\widetilde{\mathfrak{g}}_{Q/\k}$ which naturally contains $\mathfrak{g}_{Q}$ and which --contrary to $\mathfrak{g}_Q$-- \textit{does} depend on $\k$. By construction, the subspaces $\mathbf{H}^{cusp}_{Q/\k}[\mathbf{d}]$ correspond precisely to the \textit{simple} root space in $\widetilde{\mathfrak{g}}_Q[\mathbf{d}]$, see Section~\textbf{3.3}. In other words, the dimensions of the spaces of cuspidals in $\mathbf{H}_{Q/\k}$ correspond to the \textit{multiplicities} of the simple roots in $\widetilde{\mathfrak{g}}_{Q/\k}$, i.e. to the (usually infinite) Cartan matrix $\widetilde{\mathbf{A}}_{Q/\k}$ of $\widetilde{\mathfrak{g}}_Q$. \textit{Assuming} Conjecture~\ref{C:main}, one may define a \textit{$\mathbb{N}$-graded} (usually infinite) Cartan matrix $\widetilde{\mathbf{A}}_Q$ using the polynomials $C^{abs}_{Q,\mathbf{d}}$. Let $\widetilde{\mathfrak{g}}_Q$ be its associated $\mathbb{Z}$-graded Borcherds Kac-Moody algebra.

\vspace{.1in}

Maulik and Okounkov constructed in \cite{MO} a $\mathbb{Z}$-graded Borcherds Lie algebra $\mathbf{g}_Q$ using the geometry of stable envelopes in Nakajima quiver varieties. Okounkov conjectured that $\mathbf{g}_Q$ has a character precisely given (up to a global grading shift) by $\sum_{\d} A_{Q,\d}(t)z^\d$. The following conjecture is essentially a restatement of Okounkov's~:

\vspace{.1in}

\begin{conjecture}\label{C:main3} The $\mathbb{Z}$-graded Lie algebras $\mathbf{g}_Q$ and $\widetilde{\mathfrak{g}}_Q$ are isomorphic.
\end{conjecture}

\vspace{.1in}

This strongly suggests that the conjectural varieties $\mathcal{X}_{Q,\mathbf{d}}$ should be extracted from the geometry of Nakajima quiver varieties as well, but we have been unable to formulate any precise guess.

\vspace{.1in}

In a similar direction, Davison and Meinhardt defined another $\mathbb{Z}$-graded Lie algebra $\mathfrak{g}^{BPS}_{Q}$ using the theory
of (relative) cohomological Hall algebras, and showed that its graded character is (again up to a global shift in grading) given by
$\sum_{\mathbf{d}} A_{Q,\mathbf{d}}(t)z^{\mathbf{d}}$ (see \cite{DM}). Conjecture~\ref{C:main} would follow if one could prove that $\mathfrak{g}^{BPS}_{Q}$ were the positive half of a Borcherds algebra.

\vspace{.2in}

\noindent
\textbf{1.4.} \textit{Nilpotent analogues.} Assume that our quiver $Q=(I,\Omega)$ contains loops or cycles. Then the category of representations of $Q$ contains (at least) two natural Serre subcategories, that is the categories of \textit{nilpotent} and $1$-\textit{nilpotent} representations (see \cite{BSV}). We may consider the Hall algebra of the category of all nilpotent (resp. $1$-nilpotent) $Q$-modules, and the problem of counting cuspidal functions there. All the results of this paper hold, with similar proofs (and with the appropriate substitution of Kac polynomials for nilpotent Kac polynomials). Note, however, that the Maulik-Okounkov Lie algebra has not been defined in the nilpotent context (but see \cite{SV} for a the closely related cohomological Hall algebra which can be defined in the nilpotent context).

\vspace{.2in}

\noindent
\textbf{1.5.} \textit{Plethystic notation.} Throughout this paper we will use the standard plethystic notations, which we briefly recall here : let $t_1, t_2, \ldots$ be a fixed set of indeterminates, set 
$\mathbb{L}=\mathbb{Q}[[t_1, t_2, \ldots]]$ and let $\mathbb{L}^+ \subset \mathbb{L}$ be the augmentation ideal of $\mathbb{L}$. Consider the $\mathbb{Q}$-linear algebra homomorphisms 
$$\psi_l : \mathbb{L} \to \mathbb{L}, \qquad t_i \mapsto t_i^l,$$
and define inverse $\mathbb{Q}$-linear maps
$$\textup{Exp}: \mathbb{L} \mapsto \mathbb{L}^+, \qquad \textup{Log}: \mathbb{L}^+ \to \mathbb{L}$$
by the formulas
$$\textup{Exp}(f)=exp\left(\sum_{l \geq 1} \frac{1}{l}\psi_l(f)\right), \qquad \textup{Log}(f)=\sum_{l \geq 1} \frac{\mu(l)}{l} \psi_l(log(f)).$$
These satisfy the usual relations, i.e. $\textup{Exp}(f+g)=\textup{Exp}(f)\textup{Exp}(g)$, $\textup{Log}(fg)=\textup{Log}(f)+\textup{Log}(g)$, etc.
We will sometimes specifiy the variables $t_1, \ldots$ with respect to which we define the Adams operator by a subscript, as in
$\textup{Exp}_{z}$ or $\textup{Exp}_{t,z}$. Moreover, 
$$\textup{Exp}_{t_1, \ldots, t_s} \left( \sum_{n_1, \ldots, n_s} a_{n_1,\ldots, n_s} t_1^{n_1} \cdots t_s^{n_s}\right)=\prod_{n_1, \ldots n_s} \frac{1}{(1-t_1^{n_1} \cdots t_s^{n_s})^{a_{n_1, \ldots, n_s}}}.$$

\vspace{.2in}

{\centerline{\textbf{Acknowledgements}}}

\vspace{.1in}

We are grateful to B. Davison and  A. Okounkov for some stimulating discussion and correspondences, and to B. Deng and J. Xiao for explanations concerning their work \cite{DX}. The work of the first author started during his postdoctoral appointment at MIT, before being supported by the LABEX MILYON (ANR-10-LABX-0070) of Universit\'e de Lyon, within the program ``Investissements d'Avenir'' (ANR-11-IDEX-0007) operated by the French National Research Agency (ANR). The second author was partially supported by ANR (grant 13-BS01- 0001-01). 

\vspace{.2in}

\section{Borcherds-Kac-Moody algebras}

\vspace{.1in}

\noindent
\textbf{2.1.} \textit{Definition.} Let $Y$ be a free $\mathbb{Z}$-lattice equipped with a symmetric $\mathbb{Z}$-valued bilinear form $\{\,,\,\}: Y \otimes Y \to \mathbb{Z}$. Let $\Pi=\{\alpha_i, i \in I\}$ be an at most countable subset of $Y$ of $\mathbb{Q}$-linearly independent vectors, and assume that the following holds~:
$$\{\alpha_i,\alpha_j\} \leq 0 \quad \forall\; i \neq j, \qquad \{\alpha_i,\alpha_i\} \in \{2\} \cup \mathbb{Z}_{\leq 0}.$$
We will then say that $(Y,\{\;,\;\},\Pi)$ is a \textit{Borcherds datum}. We write $a_{ij}=\{\alpha_i,\alpha_j\}$ and set 
$$I^{hyp}=\{\alpha_i\;|\; a_{ii} < 0\}, \quad I^{iso}=\{\alpha_i\;|\; a_{ii} = 0\}, \quad I^{re}=\{\alpha_i\;|\; a_{ii} > 0\}.$$
Elements of $I^{hyp},I^{iso},I^{re}$ are called the \textit{hyperbolic}, resp. \textit{isotropic}, resp. \textit{real} simple roots. Elements of $I^{im}=I^{hyp} \cup I^{iso}$ are called \textit{imaginary} simple roots. We also set $A=(a_{ij})_{i,j}$ and call this the
\textit{Borcherds-Cartan matrix}. The Borcherds-Kac-Moody algebra associated to $A$ is the (generally infinite-dimensional) Lie algebra $\mathfrak{g}_A$ generated by elements $e_i,f_i,h_i, i \in I$ subject to the following set of relations
$$[h_i,h_j]=0 \qquad [h_i,e_j]=a_{ij}e_j, \qquad [h_i,f_j]=-a_{ij}f_j \qquad \forall\;i,j \in I,$$
$$[e_i,f_j]=h_i \delta_{ij}, \qquad ad^{1-a_{ij}}(e_i)e_j=ad^{1-a_{ij}}(f_i)f_j=0 \qquad \forall \;i \in I^{re}, j \in I,$$
$$[e_i,e_j]=[f_i,f_j]=0 \qquad \forall \;i,j \text{ such that }a_{ij}=0.$$

Set $\uQ=\bigoplus_{i \in I} \mathbb{Z}\alpha_i \subset Y$ and $\uQ^{\pm}=\bigoplus_i (\pm \mathbb{N})\alpha_i$.
The Lie algebra $\mathfrak{g}_A$ possesses a $\uQ$-grading 
$$\mathfrak{g}_A=\bigoplus_{\alpha} \mathfrak{g}_{\alpha}$$
with finite-dimensional graded pieces, such that $e_i \in \mathfrak{g}_{\alpha_i}, f_i \in \mathfrak{g}_{-\alpha_i}$. There is also a triangular decomposition $\mathfrak{g}_A = \mathfrak{n}_+ \oplus \mathfrak{h} \oplus \mathfrak{n}_-$, where $\mathfrak{n}_{\pm}=\bigoplus_{\alpha \in \uQ^{\pm}} \mathfrak{g}_\alpha$ and $\mathfrak{h}=\mathfrak{g}_0$. We set
$$\Delta=\{\alpha \in \uQ \backslash \{0\}\;|\; \mathfrak{g}_{\alpha} \neq \{0\}\}, \qquad \Delta^{\pm}=\Delta \cap (\pm \uQ^{\pm})$$
so that $\Delta= \Delta^+ \cup \Delta^-$. Finally the Weyl group of $\mathfrak{g}$ is the subgroup of $W\subset GL(Y)$ generated by the simple reflections $s_i : x \mapsto x- \{x,\alpha_i\} \alpha_i$, for $i \in I^{re}$. It preserves the root lattice $\uQ$ as well as the root system $\Delta$, and is in fact isomorphic, as an abstract group, to the Weyl group of the (usual) Kac-Moody algebra corresponding to the submatrix of $A$ indexed by $I^{re}\times I^{re}$. The sign representation $\epsilon : W \to \{\pm 1\}$ is the unique character of $W$ satisfying
$\epsilon(s_i)=-1$ for all $i$.

\vspace{.1in}

\noindent
\textit{Remark.} The above definition really only covers a special case of Borcherds algebras, but it will be sufficient for our purposes (see \cite{Borcherds} for details and for the general case). 

\vspace{.2in}

\noindent
\textbf{2.2.} \textit{The Borcherds denominator formula.} The enveloping algebra $U(\mathfrak{n}_+)$ inherits a $\uQ^+$-grading
from $\mathfrak{n}$. Its graded pieces are also finite dimensional, and we may define its character as the formal sum
$${ch}(U(\mathfrak{n}_+))=\sum_{\alpha} \text{dim}( U(\mathfrak{n}_+)_{\alpha}) z^{\alpha}.$$
The Borcherds character formula is an explicit expression for $ch(U(\mathfrak{n}_+))$ in terms of the Weyl group $W$ and an auxiliary term taking into account the imaginary simple roots. Enlarging $Y$ if necessary, we may fix an element $\rho \in Y$ satisfying
$\{\rho, \alpha_i\}=\frac{1}{2}a_{ii}$ for all $i \in I$. Let $\sigma$ be the collection of (unordered) tuples $\{\gamma_1, \ldots, \gamma_s\}$
of distinct, pairwise orthogonal, imaginary simple roots $\gamma _k \in I^{im}$. For $\underline{\gamma} \in \sigma$ we write
$$|\underline{\gamma}|=\sum_i \gamma_i, \qquad \epsilon(\underline{\gamma})=(-1)^{s}.$$
The Borcherds character formula now reads (see \cite{Borcherds})
\begin{equation}\label{E:BCF}
ch(U(\mathfrak{n}_+))=\left( \sum_{w \in W} \epsilon(w)z^{\rho-w\rho}w(S)\right)^{-1}
\end{equation}
where
$$S=\sum_{\underline{\gamma} \in \sigma} \epsilon(\underline{\gamma})z^{|\underline{\gamma}|}.$$

\vspace{.2in}

In \cite{KangBorcherds}, Kang defined a quantum version $U_q(\mathfrak{g}_A)$ of $U(\mathfrak{g}_A)$, defined over the field $\mathbb{C}(q)$, and proved that it is a flat deformation of $U(\mathfrak{g}_A)$. In \cite{SVdB}, Sevenhant and Van den Bergh considered the case of a complex (non root of unity) deformation parameter $q$, and proved similar results. In particular, the same Borcherds character formula (\ref{E:BCF}) holds for $U_q(\mathfrak{n}_+)$ for any $q \in \mathbb{C}^*$ which is not a root of unity.

\vspace{.2in}

\noindent
\textbf{2.3.} \textit{The Borcherds denominator formula in the presence of grading and charge.} It will be important for us to consider some slight generalizations of the above setting. 

\vspace{.1in}

Assume first that a \textit{charge} function $\mathbf{m}~:I \to \mathbb{N}, i \mapsto m_{\alpha_i}$ is given, describing the multiplicity of the simple root space $\mathfrak{g}_{\alpha_i}$. We always assume that $m_{\alpha_i}=1$ of $i \in I^{re}$, and also call the datum $(Y,\{\;,\;\},\Pi,\mathbf{m})$ a Borcherds datum. The definition of the associated Borcherds algebra $\mathfrak{g}_{A,\mathbf{m}}$ is the same as before except that we replace the generators $e_{\alpha_i}, f_{\alpha_i}$ by collections of generators $e_{\alpha_i}^{(l)}, f_{\alpha_i}^{(l)}$ for $l=1, \ldots, m_{\alpha_i}$, satisfying the same relations. The Weyl group $W$ remains unchanged, while the Borcherds character formula may now be written as follows~:
\begin{equation}\label{E:BCF2}
ch(U(\mathfrak{n}_{\mathbf{m},+}))=\left( \sum_{w \in W} \epsilon(w)z^{\rho-w\rho}w(S_{\mathbf{m}})\right)^{-1}
\end{equation}
where
$$S_\mathbf{m}=\sum_{\underline{\gamma} \in \sigma}\epsilon(\underline{\gamma})\prod_i S_{\mathbf{m},\gamma_i}$$
and
$$S_{\mathbf{m},\gamma}=m_{\gamma}z^\gamma$$
if $\gamma \in I^{hyp}$ while
$$S_{\mathbf{m},\gamma}=\sum_{l \geq 1} (-1)^l\binom{m_{\gamma}}{l}z^{l\gamma}=(1-z^{\gamma})^{m_{\gamma}}-1$$
if $\gamma \in I^{iso}$. 

\vspace{.1in}

Next, let us assume in addition that each simple root space is $\mathbb{N}$-graded, so that we have a \textit{graded charge} function
$\mathbf{m}^{\mathbb{N}}~: I \to \mathbb{N}[t]$, $\gamma \mapsto \sum_j m_{j,\gamma}t^j$. Then the associated Borcherds algebra $\mathfrak{g}^{\mathbb{N}}_{A,\mathbf{m}}$ is itself $\mathbb{N}$-graded, and the Borcherds denominator formula becomes
\begin{equation}\label{E:BCF3}
ch(U(\mathfrak{n}^{\mathbb{N}}_{\mathbf{m},+}))=\left( \sum_{w \in W} \epsilon(w)z^{\rho-w\rho}w(S^{\mathbb{N}}_{\mathbf{m}})\right)^{-1}
\end{equation}
where
$$S^{\mathbb{N}}_\mathbf{m}=\sum_{\underline{\gamma} \in \sigma}\epsilon(\underline{\gamma})\prod_i S^{\mathbb{N}}_{\mathbf{m},\gamma_i}$$
and
$$S^{\mathbb{N}}_{\mathbf{m},\gamma}=m_{\gamma}(t)z^\gamma$$
if $\gamma \in I^{hyp}$ while
$$S^{\mathbb{N}}_{\mathbf{m},\gamma}=\sum_{l_1, l_2, \ldots} (-1)^{\sum_j l_j}\prod_j\binom{m_{j,\gamma}}{l_j}t^{\sum j l_j}z^{\sum_j l_j\gamma}=\prod_j (1-t^jz^{\gamma})^{m_{j,\gamma}}-1$$
if $\gamma \in I^{iso}$. 

\vspace{.1in}

Observe that the above character formulas formally make sense even for a charge function $\mathbf{m}~: I \to \mathbb{Q}[t]$ -- we simply develop everything in power series in $z$.

\vspace{.2in}

\section{Generalities on Hall algebras and the Sevenhant-Van den Bergh theorem}

\vspace{.1in}

\noindent
\textbf{3.1.} \textit{Hall algebra.} Throughout we fix an arbitrary locally finite quiver $Q=(I,\Omega)$. For any finite field $\k$ we write $\k Q$ for the path algebra of $Q$ over $\k$, and we denote by $Rep_\k Q$ the category of finite-dimensional representations of $\k Q$. Recall that 
$$\langle\;,\;\rangle~: \mathbb{Z}^I \times \mathbb{Z}^I \to \mathbb{Z}, \qquad \langle dim(M),dim(N)\rangle = hom(M,N)-ext^1(M,N)$$
 is the Euler form. We denote by
 $$(\;,\;)~:  \mathbb{Z}^I \times \mathbb{Z}^I \to \mathbb{Z}, \qquad (\mathbf{d},\mathbf{n})=\langle \mathbf{d},\mathbf{n}\rangle + \langle \mathbf{n},\mathbf{d}\rangle$$
 the symmetrized Euler form, and by
 $$A_Q=(a_{ij})_{i,j\in I}, \qquad a_{ij}=(\epsilon_i, \epsilon_j)$$
 the associated Cartan matrix. Observe that by construction $A_Q$ is a Borcherds-Cartan matrix (with even entries on the diagonal), which does not depend on the choice of the finite field $\k$.
  
 \vspace{.1in}
 
 Let us denote by $\mathcal{M}_{\k}$ the set of isomorphism classes of objects of $Rep_{\k}Q$, and split it according to the dimension vector $$\mathcal{M}_{\k}=\bigsqcup_{\mathbf{d} \in \mathbb{N}^I} \mathcal{M}_{\k}[\mathbf{d}].$$
The Hall algebra of $Q$ over $\k$ is by definition the $\mathbb{N}^I$-graded $\C$-vector space
$$\mathbf{H}_{Q/\k}=\bigoplus_{\mathbf{d}} \mathbf{H}_{Q/\k}[\mathbf{d}],\qquad \mathbf{H}_{Q/\k}[\mathbf{d}]=\{ f : \mathcal{M}_{\k}[\mathbf{d}] \to \C\}$$
equipped with the multiplication
$$f \star g~: M \mapsto \sum_{N \subseteq M} f(M/N) g(N) v^{\langle M/N,N\rangle}$$
and the comultiplication
$$\Delta(f)(M,N)= \frac{v^{-\langle M,N\rangle}}{|Ext^1(M,N)|}\sum_{\xi \in Ext^1(M,N)}f(X_{\xi})$$
where $X_{\xi}$ denotes the extension of $M$ by $N$ corresponding to $\xi$. Here, we have set $v=|\k|^{\frac{1}{2}}$.
With the above definitions, $\mathbf{H}_{Q/\k}$ is a twisted bialgebra. Namely, if we equip the vector space $\H_{Q/\k} \otimes \H_{Q/\k}$ with the
new multiplication
$$(u \otimes w) \star (u' \otimes w')=v^{(deg(w),deg(u'))} (u \star u') \otimes (w \star w')$$
then $\Delta:\H_{Q/\k} \to \H_{Q/\k} \otimes \H_{Q/\k}$ becomes a morphism of algebras.
We define a Hermitian bilinear form $(\;,\;):{\mathbf{H}}_{Q/\k} \otimes {\mathbf{H}}_{Q/\k} \to \mathbb{C}$ by
$$(f,g)=\sum_M \frac{f(M)\overline{g(N)}}{|Aut(M)|}.$$
It is nondegenerate and has the Hopf property, i.e.
$$(a\star b,c)=(a \otimes b,  \Delta(c)), \qquad \forall\; a,b,c \in {\mathbf{H}}_{Q/\k}.$$

\vspace{.1in}

Let $I_{Q,\d}, A_{Q,\d} \in \mathbb{Q}[t]$ be the Kac polynomials counting indecomposable, resp. absolutely indecomposable representations of $Q/\k$ of dimension $\d$, see \cite{Kac}.

\vspace{.1in}

\begin{lemma}\label{L:Kacexp} We have
\begin{equation}\label{E:Kacexp}
\sum_{\mathbf{d} \geq 0} \dim(\H_{Q/\k}[\d])z^{\d}=\textup{Exp}_z\left( \sum_{\d \geq 0} I_{Q,\d}(|\k|)z^{\d}\right)=\textup{Exp}_{t,z}\left( \sum_{\d \geq 0} A_{Q,\d}(t)z^{\d}\right)_{|t=|\k|}.
\end{equation}
In particular, $\dim(\H_{Q/\k}[\d])$ is equal to the evaluation at $t=|\k|$ of some polynomial in $\mathbb{N}[t]$ for any $\d$.
\end{lemma}
\begin{proof}
Since any $\k Q$-representation decomposes in a unique fashion as a direct sum of indecomposables, the l.h.s. of (\ref{E:Kacexp})
is equal to $\prod_{\d} (1-z^{\d})^{-I_d(q)}$, which is precisely equal to the m.h.s. The second equality in (\ref{E:Kacexp}) is classical, see e.g. \cite[Theorem 4.1]{Hua}. The last statement follows from the fact, proved in \cite{HLRV}, that $A_{Q,\mathbf{n}}(t) \in \mathbb{N}[t]$ for any $\mathbf{n}$.
\end{proof}



\vspace{.2in}

\noindent
\textbf{3.2.} \textit{Cuspidal functions.} By definition, the subspace of cuspidal functions in dimension $\mathbf{d}$ is
$$\mathbf{H}_{Q/\k}^{cusp}[\mathbf{d}]=\{x \in \mathbf{H}_{Q/\k}[\mathbf{d}]\;|\; \Delta'(x)=x\otimes 1 + 1 \otimes x\}.$$
For $\mathbf{d},\mathbf{n} \in \mathbb{N}^I$ we write $\mathbf{n} \leq \mathbf{d}$ if $n_i \leq d_i$ for all $i \in I$.
Denote by $\mathbf{G}(\mathbf{d}) \subset \mathbf{H}_{Q/\k}$ the subalgebra
generated by $\bigoplus_{\mathbf{n} < \mathbf{d}}\mathbf{H}_{Q/\k}[\mathbf{n}]$. From the nondegeneracy and the Hopf
property of $(\,,\,)$ we have $\mathbf{H}_{Q/\k}^{cusp}[\mathbf{d}]=\mathbf{G}(\mathbf{d})^\perp \cap \mathbf{H}_{Q/\k}[\mathbf{d}].$ It follows
also that $\mathbf{H}^{cusp}_{Q/\k}[\mathbf{d}]$ is a minimal space of new generators of $\mathbf{H}_{Q/\k}$ in dimension $\mathbf{d}$.

\vspace{.2in}

\noindent
\textbf{3.3.} \textit{The Sevenhant-Van den Bergh theorem.} Let $Q=(I,\Omega)$ be any quiver. Set 
$$\mathcal{C}_{Q/\k}=\{\mathbf{d}\;|\; \dim(\H_{Q/\k}^{cusp}[\mathbf{d}]) \neq 0\} \subset \mathbb{N}^I,$$ 
$$Y=\mathbb{Z}^{\mathcal{C}_{Q/\k}}$$
and let $\{\tilde{\epsilon}_\d\}_{\d \in \mathcal{C}_{Q/\k}}$ stand for the canonical basis of $Y$. We equip $Y$ with a symmetric bilinear form $\{\,,\,\}~: Y \otimes Y \to \mathbb{Z}$ by setting 
$$\{\tilde{\epsilon}_\d, \tilde{\epsilon}_\mathbf{n}\}=(\mathbf{d}, \mathbf{n}).$$
We also set $\Pi=\{\tilde{\epsilon}_\d\;|\; \d \in \mathcal{C}_{Q/\k}\}$ and define a charge function 
$$\mathbf{m}~: \mathcal{C}_{Q/\k} \to \mathbb{N}, ~\mathbf{d} \mapsto \dim(\H^{cusp}_{Q/\k}[\mathbf{d}]).$$
 As shown by Sevenhant and Van den Bergh, $(Y,\{\;,\;\},\Pi,\mathbf{m})$ is a Borcherds datum.  Let $\widetilde{A}_{Q/\k}$ be the Borcherds-Cartan matrix associated to the quadruple $(Y,\{\,,\,\}, \Pi, \mathbf{m})$, and let $\widetilde{\mathfrak{g}}_{Q/\k}, \widetilde{\mathfrak{n}}_{Q/\k}$ be its associated generalized Kac-Moody algebra, resp. its positive nilpotent subalgebra\footnote{in an effort to unburden the notation, we drop the subscript $+$ in $\mathfrak{n}_+$} as in \textbf{3.3.} Via the projection map $\pi: Y \to \mathbb{Z}^I, \epsilon_\d \mapsto \mathbf{d}$, we may consider $U_v(\widetilde{\mathfrak{n}}_{Q/\k})$ as a $\mathbb{N}^I$-graded algebra. 

\vspace{.1in}

\begin{theorem}[\cite{SVdB}]\label{T:SvdB} There is an isomorphism of $\mathbb{N}^I$-graded algebras $\H_{Q/\k} \simeq U_v(\widetilde{\mathfrak{n}}_{Q/\k})$. 
\end{theorem}

\vspace{.1in}

\noindent
\textit{Remark}. The above isomorphism is in fact an isomorphism of Hopf algebras but we won't need it.

\vspace{.2in}

\section{Counting cuspidal functions}

\vspace{.1in}

\noindent
\textbf{4.1.} \textit{Proof of Theorem~\ref{T:main}}. We will prove the polynomiality of $C_{Q,\d}(q)$ inductively by comparing the
Borcherds character formula (\ref{E:BCF}) and the dimension formula (\ref{E:Kacexp}) in terms of Kac I-polynomials. For $i \in I$ 
we have $C_{Q,\epsilon_i}=q^{g_i},$ where $g_i$ is the number of edge loops at the vertex $i$. Let us now fix some $\d \in \mathbb{N}^I$
and assume that there exist polynomials $C_{Q,\mathbf{n}} \in \mathbb{Q}[t]$ such that $\dim(\H^{cusp}_{Q/\k}[\mathbf{n}])=C_{Q,\mathbf{n}}(|\k|)$ for any $\k$ and any $\mathbf{n} < \mathbf{d}$. 

\vspace{.1in}

Let $A$ be the (usual, i.e. $I \times I$) Borcherds-Cartan matrix associated to $Q$, and let $W$ be its Weyl group acting on $\mathbb{Z}^I$. Let $\widetilde{A}_{Q/\k}^{< \d}$ be the submatrix of  $\widetilde{A}_{Q/\k}$ obtained by keeping only the rows and columns indexed by $\mathcal{C}_{<\mathbf{d}}=\{\mathbf{n} \in \mathcal{C}_{Q/\k} \;|\; \mathbf{n} <\d\}$, and let $\widetilde{\mathfrak{g}}^{<\d}_{Q/\k}\subset \widetilde{\mathfrak{g}}_{Q/\k}$ be the associated Borcherds subalgebra. We likewise denote by $Y_{Q/\k}^{<\mathbf{d}}, \widetilde{\mathfrak{n}}^{<\d}_{Q/\k},\ldots$ the corresponding root lattice, nilpotent subalgebra, etc. 
Recall that the Weyl group of a Borcherds-Kac-Moody algebra is generated by simple reflections $s_i$ corresponding to real roots
(hence $W^{<\d}=W$ for all $\d$), and that the pairing on $Y$ factors through the projection $\pi~:Y \to \mathbb{Z}^I$. It follows that the projection $\pi$ is $W$-equivariant.

\vspace{.1in}

By construction and by Theorem~\ref{T:SvdB} we have
\begin{equation}\label{E:proof1}
\dim(\H^{cusp}_{Q/\k}[\d])=\dim(\H_{Q/\k}[\d])-\dim (U(\widetilde{\mathfrak{n}}^{<\d}_{Q/\k})[\d]).
\end{equation}
By Lemma~\ref{L:Kacexp}, $dim(\H_{Q/\k}[\d])$ is the evaluation at $t=|\k|$ of some universal polynomial in $\mathbb{N}[t]$, hence we just have
to show that the same holds for $dim (U(\widetilde{\mathfrak{n}}^{<\d}_{Q/\k})[\d]).$ This will be a consequence of the Borcherds character formula (\ref{E:BCF}) for $ch(U(\widetilde{\mathfrak{n}}^{<\d}_{Q/\k}))$. Indeed, we claim that $dim (U(\widetilde{\mathfrak{n}}^{<\d}_{Q/\k})[\mathbf{n}])$ is a polynomial in $|\k|$ for all $\mathbf{n} \in \mathbb{N}^I$. To prove this, it is in turn enough to show that
$${X}_{Q/\k}^{<\d}=\sum_{w \in W} \epsilon(w)z^{\rho^{<\d}-w\rho^{<\d}}w({S}_{Q/\k}^{<\d})$$
is itself polynomial in $|\k|$, i.e. that for any $\beta$, the coefficient of $z^{\beta}$ is polynomial in $|\k|$ (indeed, it follows from the general theory that the support of $X_{Q/\k}^{<\d}$ is in $\mathbb{N}^I$), and thus it is enough to show that $S_{Q/\k}^{<\d}$ is polynomial in $|\k|$. Using the definition, it is easy to see that $S_{Q/\k}^{<\d}=\sum_{\underline{\beta}}X_{\underline{\beta}}$
where $\underline{\beta}$ runs among all tuples $\{\beta_1, \beta_2, \ldots\}$ of distinct imaginary roots in $\mathbb{N}^I$ which are mutually orthogonal, where we have set
$$X_{\underline{\beta}}= \prod_i X_{\beta_i}, $$ 
with $X_{\beta}=0$ unless $\beta < \d$ and otherwise
$$X_{\beta}=-C_{Q,\beta}(|\k|)z^{\beta}, \qquad \text{if}\; (\beta,\beta)<0,$$
$$X_{\beta}=(1-z^\beta)^{C_{Q,\beta}(|\k|)}-1, \qquad \text{if}\; (\beta,\beta)=0.$$
Since any element of $\mathbb{N}^I$ may be decomposed as a sum of elements of $\mathbb{N}^I$ in a finite number of ways, 
we deduce that $S_{Q/\k}^{<\d}$ is polynomial in $|\k|$ as wanted. It is clear from the above that $C_{Q,\d}(t) \in \mathbb{Q}[t]$. The induction step is complete. Note that the above counting argument only uses the symmetrized Euler form and the Kac polynomials $I_{Q,\d}(t)$, both of which are independent of the orientation of the quiver. Theorem~\ref{T:main} is proved.\qed

\vspace{.2in}

\noindent
\textbf{4.2.} \textit{Proof of Proposition~\ref{P:firstprop}.} The first statement follows from the fact that $(\mathbb{Z}^I, \{\;,\;\},\mathcal{C}_{Q/\k})$ is a Borcherds datum. Indeed, this implies that for any $\mathbf{d} \in \mathcal{C}_{Q/\k} \backslash \Pi_0$ we have $(\mathbf{d},\epsilon_i) \leq 0$ for any $i \in I$. But then $(\d,\d) \leq 0$. We now prove the second statement. We will use the following slight extension, whose proof is given in the appendix, of Kac's conjecture proved by Hausel (see \cite{HauselKacConj})~: let $\mathfrak{g}_{Q^{re}}$ be the Kac-Moody algebra associated to the subquiver $Q^{re}=(I^{re},\Omega^{re})$ of $Q$ consisting of the real vertices and the edges between them. Then for any $\d$, $A_{Q,\d}(0)=
\dim(\mathfrak{g}_{Q^{re}}[\d])$. In particular, $A_{Q,\d}(0)=0$ whenever $\d \not\in \mathbb{N}^{I^{re}}$.
This implies that
\begin{equation}\label{E:kacatt=0}
\begin{split}
\textup{Exp}_{t,z}\left( \sum_{\d} A_{Q,\d}(t)z^{\d}\right)_{|t=0}&=\textup{Exp}_z\left( \sum_{\d} A_{\d}(0) z^\d\right)\\&=\textup{Exp}_z\left( \sum_{\d} \dim(\mathfrak{g}_{Q^{re}}[\d]) z^\d\right)\\ &=ch(U(\mathfrak{n}_{Q^{re}})).
\end{split}
\end{equation}
We now argue by induction. Let us fix $\mathbf{d}$ and assume that $C_{Q,\mathbf{n}}(0)=0$ for any $\mathbf{n} \in \mathbb{N}^I \backslash \Pi_0$, $\mathbf{n} < \mathbf{d}$.
By combining (\ref{E:Kacexp}), (\ref{E:kacatt=0}) and (\ref{E:proof1}) we see that $C_{Q,\d}(0)=\dim(U(\mathfrak{n}_{Q^{re}})[d])-\dim(U(\widetilde{\mathfrak{n}}^{<\d}_{Q/\k}[d]))_{|t=0}$, where we view $\dim(U(\widetilde{\mathfrak{n}}^{<\d}_{Q/\k}[d]))$ as a polynomial in $t$ (see the proof of Theorem~\ref{T:main}). From the induction hypothesis, we have $(S^{<\d}_{Q/\k})_{|t=0}
=S_{Q^{re}}$ and thus $\dim(U(\widetilde{\mathfrak{n}}^{<\d}_{Q/\k}[d]))_{|t=0}=\dim(U(\mathfrak{n}_{Q^{re}})[\d])$. Hence $C_{Q,\d}(0)=0$ as wanted.
\qed

\vspace{.1in}

\noindent
\textit{Remark.} From the defining relations for $C^{abs}_{Q,\d}(t)$ in terms of $C_{Q,\d}(t)$ it is easy to deduce that $C^{abs}_{Q,\d}(0)=0$ for all $\d \not\in \Pi_0$.

\vspace{.2in}

\noindent
\textbf{4.3.} \textit{Counting absolutely cuspidal functions.} As shown by the examples given in the introduction, the polynomials $C_{Q,\d}(t)$ are in general neither integral, nor positive. We now replace, in the above construction, the Kac I-polynomials with the A-polynomials, and work in the setting of $\mathbb{N}$-\textit{graded} Borcherds algebras. In other words, we replace the character formulas
\begin{equation}\label{E:gradedproof0}
\pi\left( ch(U(\widetilde{\mathfrak{n}}_{Q/\k}))\right)=\textup{Exp}_{z}\left( \sum_{\d >0} I_{Q,\d}(t)z^{\d}\right)_{|t=|\k|} \qquad \forall\; \k
\end{equation}
by the graded analog
\begin{equation}\label{E:gradedproof1}
\pi\left( ch(U(\widetilde{\mathfrak{n}}^{\mathbb{N}}_{Q}))\right)=\textup{Exp}_{t,z}\left( \sum_{\d >0} A_{Q,\d}(t)z^{\d}\right).
\end{equation}
Note that the r.h.s of (\ref{E:gradedproof1}) is equal to the r.h.s of (\ref{E:gradedproof0}) before evaluation at $t=|\k|$ (see e.g. Lemma~\ref{L:Kacexp}). Of course, the \textit{existence} of the $\mathbb{N}$-graded Borcherds algebra $\widetilde{\mathfrak{g}}^{\mathbb{N}}_Q$ is still conjectural --it is the essence of Conjecture~\ref{C:main}--
but we may \textit{assume} that it is associated to a data $(Y,\{\,,\,\},\Pi,\mathbf{m}^{\mathbb{N}})$
where $Y=\mathbb{Z}^{\mathbb{N}^I}$, $\{\;,\;\}$ is induced by the Euler form $(\;,\;)$ on $\mathbb{Z}^I$, $\Pi=\{\tilde{\epsilon}_{\d}\;|\; \d \in \mathbb{N}^I\}$, and $\mathbf{m}^{\mathbb{N}}~:\Pi \to \mathbb{Q}[t]$ is a charge function, and \textit{then} determine the values of this charge function. Note that by the exact same argument as in \textbf{4.1.}, $\mathbf{m}^{\mathbb{N}}$ is indeed uniquely determined, i.e. there exists a unique family of rational polynomials $C^{abs}_{Q,\d} \in \mathbb{Q}[t]$ for which the (formal) character of the $\mathbb{N}$-graded Borcherds algebra associated to $(Y,\{\,,\,\},\Pi,\mathbf{m}^{\mathbb{N}})$ with $\mathbf{m}^{\mathbb{N}}(\tilde{\epsilon}_\mathbf{d})=C^{abs}_{Q,\mathbf{d}}(t)$ satisfies (\ref{E:gradedproof1}).

\vspace{.1in}

As it turns out, there is a rather simple relationship between the families of polynomials $C_{Q,\d}(t)$ and $C^{abs}_{Q,\d}$. More precisely,

\vspace{.1in}

\begin{proposition}\label{P:relcusp} The following hold~:
\begin{enumerate}
\item[i)] We have $C_{Q,\d}(t)=C^{abs}_{Q,\d}$ for any $\d \in \mathbb{N}^I$ satisfying $(\d,\d)<0$,
\item[ii)] For any $\mathbf{d} \in (\mathbb{N}^I)_{prim}$ such that $(\d,\d)=0$ we have
$$\textup{Exp}_z\left(\sum_{l \geq 1} C_{Q,l\mathbf{d}}(t)z^{l}\right)=\textup{Exp}_{t,z}\left(\sum_{l \geq 1} C^{abs}_{Q,l\mathbf{d}}(t)z^{l}\right).$$
\end{enumerate}
\end{proposition}
\begin{proof} Let $\widetilde{\mathfrak{g}}^\mathbb{N}_{Q}$ be the $\mathbb{N}$-graded Borcherds algebra associated to the data $(Y,\{\;,\;\},\Pi,\mathbf{m}^{\mathbb{N}})$, with $\mathbf{m}^{\mathbb{N}} (\mathbf{d})=C^{abs}_{Q,\d}(t)$, where $C^{abs}_{Q,\d}(t)$ are determined by i) and ii) above.
We have to show that the $\mathbb{N} \times \mathbb{N}^I$-graded character of $U(\widetilde{\mathfrak{n}}_Q^{\mathbb{N}})$ is equal to the r.h.s of (\ref{E:gradedproof1}). Using
(\ref{E:gradedproof0}), it is enough to show that
$$\pi\left(ch(U(\widetilde{\mathfrak{n}}^{\mathbb{N}}_Q))\right)_{|t=|\k|}=\pi\left(ch(U(\widetilde{\mathfrak{n}}_{Q/\k}))\right) \qquad \forall\;\k.$$
Allowing for zero multiplicities, we may replace $\mathcal{C}_{Q/\k}$ by $\mathbb{N}^I$ and view both $\widetilde{\mathfrak{g}}_{Q/\k}$ and $\widetilde{\mathfrak{g}}_Q^{\mathbb{N}}$ as Borcherds algebras built on the same lattice $Y=\mathbb{Z}^{\mathbb{N}^I}$ and same set of simple roots $\Pi=\{\widetilde{\epsilon}_{\d}\;|\; \d \in \mathbb{N}^I\}$ (but with different charge functions). Let $\{\tilde{\epsilon}_{\d}\}$ be the canonical basis of $Y$, and let $Y'$ be the quotient of $Y$ by the $\mathbb{Z}$-submodule generated by the collection of elements $\tilde{\epsilon}_{l\mathbf{d}}-l\tilde{\epsilon}_{\mathbf{d}}$, for $l \in \mathbb{N}$ and $\d \in \mathbb{N}^I$. There is a canonical identification $Y' \simeq \mathbb{Z}^{(\mathbb{N}^I)_{prim}}$. The projection map $\pi~:Y \to \mathbb{Z}^I$ factors through the quotient
$\pi'~: Y \to Y'$, and it is enough to prove that
$$\pi'\left(ch(U(\widetilde{\mathfrak{n}}^{\mathbb{N}}_Q))\right)_{t=|\k|} =\pi'\left(ch(U(\widetilde{\mathfrak{n}}_{Q/\k}))\right) \qquad \forall \;\k.$$
With obvious notations, the above equality boils down to
\begin{equation}\label{E:proofgraded2}
\pi'(({S}_Q^{\mathbb{N}}))_{t=|\k|}=\pi'({S}_{Q/\k}).
\end{equation}
 To prove (\ref{E:proofgraded2}), we will split the sums defining $S_Q^{\mathbb{N}}, S_{Q/k}$ according to the type of simple roots simple roots $\gamma_i$ occuring in some $\underline{\gamma} \in \sigma$. More precisely, we will say that a collection of orthogonal imaginary roots $\underline{\gamma}=(\gamma_i)_i \in \sigma$ is \textit{of primitive type}
 $\{\beta_1, \ldots, \beta_n\}$ for some $\beta_1, \ldots, \beta_n \in (\mathbb{N}^I)_{prim}$ if every $\gamma_i$ is an integer multiple of some $\beta_j$. Now let us fix a primitive type $\{\beta_1, \ldots, \beta_n\}$ and let ${}_{\beta}S_Q^{\mathbb{N}}, {}_\beta S_{Q/\k}$ stand for the partial sums over all $\underline{\gamma} \in \sigma$ of primitive type $\{\beta_1,\ldots, \beta_n\}$. Observe that any finite collection of integer multiples of the $\beta_i$ belongs to $\sigma$ since the form $\{\;,\;\}$ is pulled back from $\mathbb{Z}^I$. Hence we have
 $${}_\beta S_{Q/\k}=\prod_{i=1}^nZ_{\beta_i}, \qquad {}_\beta S_Q^{\mathbb{N}}=\prod_{i=1}^nZ^{\mathbb{N}}_{\beta_i}(t),$$
 where for $(\beta_i, \beta_i)<0$
 $$Z_{\beta_i}=\prod_{l \geq 1}(1-C_{Q,l\beta_i}(|\k|)z^{l\beta_i})-1, $$
 $$Z^{\mathbb{N}}_{\beta_i}=\prod_{l \geq 1}(1-C^{abs}_{Q,l\beta_i}(t)z^{l\beta_i})-1$$
 while for $(\beta_i,\beta_i)=0$
 $$Z_{\beta_i}=\prod_{l \geq 1}(1-z^{l\beta_i})^{C_{Q,l\beta_i}(|\k|)}-1=\textup{Exp}_z\left(-\sum_{l\geq 1} C_{Q,l\beta_i}(t)z^{l\beta_i}\right)_{|t=|\k|}-1$$
 and, writing $C^{abs}_{Q,l\beta_i}=\sum_p a_{p,l\beta_i}t^p$,
 $$Z^{\mathbb{N}}_{\beta_i}=\prod_{l \geq 1}\prod_{p \geq 0}(1-t^pz^{l\beta_i})^{a_{p,l\beta_i}}-1=\textup{Exp}_{t,z}\left(-\sum_{l\geq 1} C^{abs}_{Q,l\beta_i}(t)z^{l\beta_i}\right)-1.$$
 Conditions i) and ii) of Proposition~\ref{P:relcusp} precisely guarantee the equality $\pi'({}_{\beta}S_Q^{\mathbb{N}})_{t=|\k|}= \pi'({}_\beta S_{Q/\k})$. Note that $\pi'(z^{l\beta_i})=\pi'((z^{\beta_i})^l)$. Summing up over all primitive types yields (\ref{E:proofgraded2}). We are done.
\end{proof}

\vspace{.2in}

\noindent
\textbf{4.4.} \textit{Integrality of $C^{abs}_{Q,\d}$.} In this section we prove that the polynomials $C^{abs}_{Q,\d}$ defined by (\ref{E:gradedproof1}) actually belong to $\mathbb Z[t]$. We will first deal with the ones associated to isotropic vectors. Recall that $\mathcal{C}_{Q/\k}$ denotes the set of cuspidal dimensions, and that $\mathcal{C}_{Q/\k}^{iso}$, $\mathcal{C}_{Q/\k}^{hyp}$ denote the subsets of $\d$ satisfying respectively $(\d,\d)=0$, $(\d,\d)<0$. We know that $\mathcal{C}_{Q/\k}\backslash I^{re}=\mathcal{C}_{Q/\k}^{im}=\mathcal{C}_{Q/\k}^{iso}\sqcup\mathcal{C}_{Q/\k}^{hyp}$.

\begin{lemma}
Consider $\d\in \mathcal{C}_{Q/\k}^{iso}$ and a decomposition $$
\d=\sum_{1\le p\le r}n_p\d_p+\sum_{1\le q\le s}n'_q\epsilon_{i_q}$$
where $\d_p\in\mathcal{C}_{Q/\k}^{im}$, $i_q\in I^{re}$ and $n_p,n'_q\in\mathbb N_{>0}$. Set $\d^{re}=\sum_qn'_qi_q$. Then \begin{enumerate}[label=(\roman*)]
\item $(\d_p,\d_{p'})=0$ for every $p,p'$;
\item $(\d,\d_p)=(\d,i_q)=(\d_p,i_q)=0$;
\item $(\d^{re},\d^{re})=0$.
\end{enumerate}
\end{lemma}

\begin{proof}
We know that $(Y=\mathbb Z^{\mathcal{C}_{Q/\k}},\{\;,\;\},\Pi,\mathbf m=(C_{Q,\d})_\d)$ is a Borcherds datum, hence the elements of the sum$$
0=(\d,\d)=\sum_qn'_q(\epsilon_{i_q},\d)+\sum_{p,p'}n_pn_{p'}(\d_p,\d_{p'})+\sum_{p,q}n_pn'_{q}(\d_p,\epsilon_{i_q})$$
are all nonpositive, and thus equal to zero. Similarly,$$
0=(\d,\d)=\sum_pn_p(\d,\d_p)+\sum_qn'_q(\d,\epsilon_{i_q})=\sum_pn_p(\d,\d_p)$$
implies $(\d,\d_p)=0$. The equality \textit{(iii)} then follows from$$
(\d^{re},\d^{re})=(\d-\sum_pn_p\d_p,\d-\sum_pn_p\d_p)=0.$$
\end{proof}

We can now turn to the proof of Theorem~\ref{T:main2}.
\begin{proof}[Proof of Theorem~\ref{T:main2}]
Consider $\d\in\mathcal{C}_{Q/\k}^{iso}$, and $$
D=\sum_{1\le p\le r}n_p\d_p+\sum_{1\le q\le s}n'_q\epsilon_{i_q}\in\pi^{-1}(\d).$$
From the previous lemma we get$$
\dim U(\mathfrak n_{\mathbf m})[D]=\dim U(\mathfrak n_{Q^{re}})[\d^{re}]\times\prod_p\dim U(\mathfrak n_{\mathbf m})[n_p\d_p]$$
where $\d^{re}=\sum_{1\le q\le s}n'_q\epsilon_{i_q}$. Hence, summing up over all $D \in \pi^{-1}(\d)$ and then over all $\d \in \mathcal{C}_{Q/\k}^{iso}$ we have$$
\left\{\textup{Exp}_{t,z}\left(\sum_\d A_\d(t)z^\d\right)\right\}_{iso}=\left\{ch(U(\mathfrak n_{Q^{re}}))\times\textup{Exp}_{z}\left(\sum_{\d\in \mathcal{C}_{Q/\k}^{iso}} C_\d z^\d\right)\right\}_{iso}$$
where $
\{\sum_{\d}b_\d z^\d\}_{iso}=\sum_{\d\in Y^{iso}}b_\d z^\d$. Using~\ref{P:relcusp}, we obtain$$
\left\{\textup{Exp}_{t,z}\left(\sum_\d A_\d(t)z^\d\right)\right\}_{iso}=\left\{ch(U(\mathfrak n_{Q^{re}}))\times\textup{Exp}_{t,z}\left(\sum_{\d\in \mathcal{C}_{Q/\k}^{iso}} C^{abs}_\d z^\d\right)\right\}_{iso}$$
and we get that $C^{abs}_\d\in\mathbb Z[t]$ by induction on $\d\in \mathcal{C}_{Q/\k}^{iso}$ since the left hand side belongs to $\mathbb N[t][[z_i]]_{i\in I}$. Indeed we have $\d_p,\epsilon_{i_q}<\d$ in any nontrivial decomposition $D$, i.e.\ when $D\neq1.\d$.

Now consider $\mathbf m^{\mathbb N}_\d=C^{abs}_{Q,\d}(t)$ formally defined (hence a priori $\in\mathbb Q[t]$) by$$
\pi\left(ch(U(\mathfrak n_{\mathbf m}^{\mathbb N}))\right)=\textup{Exp}_{z,t}\left(\sum_\d A_\d(t)z^\d\right)$$
with the notations of section 2.3. We know that $ch(U(\mathfrak n_{\mathbf m}^{\mathbb N}))$ is obtained from$$
ch(U(\mathfrak n_{\mathbf 1}))=\sum_{D\in Y}d_Dz^D$$
by replacing $z^\d$ by $m_\d(t)z^\d$ if $\d\in\mathcal{C}_{Q/\k}^{hyp}$ and by
\begin{equation}\label{E:isomult}
\prod_j(1-t^jz^\d)^{m_{\d,j}}-1
\end{equation}
if $\d\in\mathcal{C}_{Q/\k}^{iso}$. From the isotropic case, (\ref{E:isomult}) belongs to $\mathbb N[t][[z_i]]_{i\in I}$.
Since $d_D=1$ if $D=1.\d$, we can again conclude by induction on $\d\in\mathcal{C}_{Q/\k}^{hyp}$ that $C^{abs}_{Q,\d}\in\mathbb Z[t]$.
\end{proof}

\vspace{.2in}

\noindent
\textbf{4.5.} \textit{An example : totally negative quivers.} Here we provide the proof of Theorem~\ref{P:totneg}. 
From the very definition of a totally negative quiver it follows that, for any $\k$, all simple roots of $\widetilde{\mathfrak{g}}_{Q/\k}$ are hyperbolic. This implies the first statement and that
$\widetilde{\mathfrak{n}}_{Q/\k}$ is a free Lie algebra, and thus that $U(\widetilde{\mathfrak{n}}_{Q/\k})$ is a free associative algebra. In this situation we have
$$\sum_{\d} \dim(\mathbf{H}_{Q/\k}[\d])z^{\d}=\sum_{\d}\dim(U(\widetilde{\mathfrak{n}}_{Q/\k}[\d]))z^{\d}=\left(1-\sum_{\d}C_{Q,\d}(|\k|)z^{\d}\right)^{-1}.$$ 
From the relations
$$\sum_{\d}\dim(\mathbf{H}_{Q/\k}[\d])z^{\d}=\textup{Exp}_z\left(\sum_{\d>0} I_{Q,\d}(|\k|)z^{\d}\right)=\textup{Exp}_{t,z}\left(\sum_{\d>0} A_{Q,\d}(t)z^{\d}\right)_{|t=|\k|}$$
we deduce Theorem~\ref{P:totneg} for the polynomials $C_{Q,\d}(t)$. \qed

\vspace{.2in}

\noindent
\textbf{4.6.} \textit{Remarks.} We end this note with a pair of comments~:

\vspace{.05in}

 i) In the case of a smooth projective curve $X$, the Langlands correspondence sets up a bijection between ($l$-adic) cuspidal functions (of rank $r$ ) and irreducible ($l$-adic) representations of $\pi_1^{arith}(X)$ of dimension $r$. Similarly, there is a bijection between ($l$-adic) absolutely cuspidal functions (of fixed rank $r$) and irreducible ($l$-adic) representations of $\pi_1^{arith}(X)$ of dimension $r$ whose restriction to $\pi_1^{geom}(X)$ remains irreducible. This yields a relation between the dimensions of the 
 spaces of cuspidal and absolutely cuspidal functions, somewhat similar to the relation between the polynomials $A_{Q,\d}(t)$ and $I_{Q,\d}(t)$, see Lemma~\ref{L:Kacexp}, or between the polynomials $C_{Q,\d}(t)$ and $C^{abs}_{Q,\d}(t)$ \textit{along a line spanned by an isotropic cuspidal dimension}. This suggests that there might be a `spectral' parametrization of cuspidal functions for quivers only in the case of \textit{isotropic} cuspidal dimensions.

\vspace{.05in}

ii) Again in analogy with the Langlands correspondence, it seems natural to try to \textit{construct explicitly} cuspidal functions for a quiver $Q$ and some dimension vector $\d$; likewise, one may try to construct some perverse sheaf lift of a suitable basis
of the space of cuspidal functions $\mathbf{H}^{cusp}_{Q/\k}$. Note that there are no natural Hecke operators in the context of quivers, and hence \textit{a priori} no preferred choice of basis in $\mathbf{H}^{cusp}_{Q/\k}$. We hope to come back to these questions in the future. 

\vspace{.2in}

\appendix

\section{A trivial variant of Kac's conjecture}

\vspace{.1in}

\begin{proposition} Let $Q=(I,\Omega)$ be any quiver, and let $Q^{re}=(I^{re},\Omega^{re})$ be the full subquiver of $Q$ whose vertices are real. Then for any dimension vector $\d \in \mathbb{N}^I$ we have
$$A_{Q,\d}(0)=\dim(\mathfrak{g}_{Q^{re}}[\d]).$$
\end{proposition}
\begin{proof} It is enough to show that $A_{Q,\d}(0)=0$ whenever $\mathbf{d}$ is not supported on $I^{re}$ (and then we are reduced to the usual Kac conjecture proved by Hausel).
Let $H_{Q,\d}(t) \in \mathbb{N}[t]$ be the polynomial counting all isomorphism classes of representations of dimension $\d$. The relation to $A_{Q,\mathbf{n}}(t)$ reads
\begin{equation}\label{E:app1}
\sum_{\d} H_{Q,\d}(t)z^{\d}=\textup{Exp}_{t,z}\left( \sum_{\mathbf{d}} A_{Q,\d}(t)z^{\mathbf{d}}\right).
\end{equation}
Evaluating (\ref{E:app1}) at $t=0$, we see that $A_{Q,\d}(0)=0$ for all $\d \not\in \mathbb{N}^{I^{re}}$ if and only if $H_{Q,\d}(0)=0$ for all $\d \not\in \mathbb{N}^{I^{re}}$. Writing
$H_{Q,\d}(t)=c_{Q,\d} + t K_{Q,\d}(t)$ with $c_{Q,\d} \in \mathbb{N}$ and $K_{Q,\d}(t) \in \mathbb{N}[t]$ we immediately see that $H_{Q,\d}(0)=0$ of and only if $q\;|\; H_{Q,\d}(q)$ for all prime powers $q$. Let us now fix a finite field $\k$ and $\d \not\in \mathbb{N}^{I^{re}}$. Let $i \in I \backslash I^{re}$ be an imaginary vertex for which $\d_i \neq 0$, and let $h \in \Omega$ be an edge loop at $i$. Define an action of the additive group $\mathbb{G}_a(\k)=(\k,+)$ on the set $Ind_{Q,\d,\k}$ of isomorphism classes of indecomposable representations of dimension $\d$ by setting 
$$\lambda \cdot (x_\gamma)_{\gamma \in \Omega} =(x_{\gamma} + \lambda \delta_{\gamma,h} Id_{\k^{\d_i}})_{\gamma \in \Omega}.$$
Considering the eigenvalues of $x_h$, we see that this action is free. This proves that $q\;|\; H_{Q,\d}(q)$ as wanted.
\end{proof}

\vspace{.1in}

\noindent
\textit{Remark.} A more interesting variant of Kac's conjecture in the context of quivers with edge loops is proved in \cite{BSV} : the constant term of the \textit{nilpotent} Kac polynomial
$A^{1}_{Q,\mathbf{d}}(t)$ is equal to the multiplicity of the root $\mathbf{d}$ in the \textit{generalized} Borcherds algebra $\mathfrak{g}_Q$ associated to $Q$ in \cite{Bozec}.

\vspace{.2in}

\small{
\noindent
T. Bozec, \texttt{bozec@math.univ-lyon1.fr},\\
Institut Camille Jordan - Universit\'e Lyon 1, B\^at. Braconnier,\\
43 Boulevard du 11 Novembre 1918, 69622 Villeurbanne cedex

\vspace{.1in}
 
\noindent
O. Schiffmann, \texttt{olivier.schiffmann@math.u-psud.fr},\\
D\'epartement de Math\'ematiques, Universit\'e de Paris-Sud Paris-Saclay, \\
B\^at.~425, 91405 Orsay Cedex}

\vspace{.2in}

\begin{thebibliography}{99}

\bibitem[B1]{Borcherds}
R. Borcherds, \emph{Generalized Kac-Moody algebras}, Journal of Algebra, \textbf{115}, 501--512 (1988).

\bibitem[B2]{Bozec}
T. Bozec, \emph{Quivers with loops and generalized crystals}, Compositio Mathematica, \textbf{152} (10), 1999--2040 (2016).

\bibitem[BSV]{BSV}
T. Bozec, O. Schiffmann, E. Vasserot, \emph{On the number of points of nilpotent quiver varieties over finite fields}, arXiv:1701.01797 (2017).

\bibitem[D]{Deligne}
P. Deligne, \emph{Comptage de faisceaux l-adiques},
in De la g\'eom\'etrie aux formes automorphes (I) (en l'honneur du soixanti\`eme anniversaire de G\'erard Laumon, Ast\'erisque. \textbf{369} (2015) 285--312.

\bibitem[DM]{DM}
B. Davison, S. Meinhardt, \emph{Cohomological Donaldson-Thomas theory of a quiver with potential and quantum enveloping algebras}, preprint arXiv:1601.02479 (2016)

\bibitem[DX]{DX}
B. Deng, J. Xiao, \emph{A new approach to Kac's theorem on representations
of valued quivers}, Math. Z. \textbf{245}, 183--199 (2003).

\bibitem[H]{HauselKacConj}
T. Hausel, \emph{Kac's conjecture from Nakajima quiver varieties}, Invent. Math. \textbf{181} (2010), 21--37.

\bibitem[HLRV]{HLRV}
T. Hausel, E. Letellier, F. Rodriguez-Villegas, \emph{Positivity for Kac polynomials and DT-invariants of quivers.}, Ann. of Math. (2) \textbf{177} (2013), no. 3, 1147--1168.

\bibitem[H]{Hua}
J. Hua, \emph{Counting representations of quivers over finite fields.} J. Algebra \textbf{226} (2000), no. 2, 1011--1033.

\bibitem[HX]{HX}
 J. Hua, J. Xiao, \emph{On Ringel-Hall algebras of tame hereditary algebras.}
Algebr. Represent. Theory \textbf{5} (2002), no. 5, 527--550. 

\bibitem[K1]{Kac}
V. Kac, \emph{Infinite root systems, representations of graphs and invariant theory}, Invent. Math.
\textbf{56}, no. 1, 57--92, (1980).

\bibitem[K2]{KangBorcherds}
S.-J. Kang, \emph{Quantum deformations of generalized Kac-Moody algebras and their modules.}
J. Algebra \textbf{175} (1995), no. 3, 1041--1066. 

\bibitem[K3]{Kont}
M. Kontsevich, \emph{Notes on motives in finite characteristic}, arXiv:0702206 (2007).

\bibitem[L]{Lafforgue}
L. Lafforgue, \emph{Chtoucas de Drinfeld et correspondance de Langlands},
Invent. Math. \textbf{147} (1), 1--242, (2002). 

\bibitem[MO]{MO}
D. Maulik, A. Okounkov, \emph{Quantum groups and quantum cohomology}, arXiv:1211.1287 (2012).

\bibitem[M]{Moz}
S. Mozgovoy, \emph{Motivic Donaldson-Thomas invariants and McKay correspondence}, arXiv:1107.6044 (2011).

\bibitem[O]{Okounkovconj}
A. Okounkov, \emph{On some interesting Lie algebras}, Conference in honor of Victor Kac, IMPA, June 2013, talk available at https://www.youtube.com/watch?v=H8rCJ7ls1K4

\bibitem[SV]{SV}
O. Schiffmann, E. Vasserot, \emph{Cohomological Hall algebras of quivers: generators}, arXiv:1705.07488 (2017).

\bibitem[SVdB]{SVdB}
B. Sevenhant, M. Van den Bergh, \emph{A relation between a conjecture of Kac and the structure of the Hall algebra}, J. Pure Appl. Algebra \textbf{160} (2001), 319--332.

\end{thebibliography}
\end{document}